\documentclass[reqno,12pt]{amsart}
\usepackage{amsmath,amssymb,amsthm,graphicx,mathrsfs,url}
\usepackage{wrapfig}
\usepackage{enumitem}
\usepackage{mathtools}
\usepackage[utf8]{inputenc}
 \usepackage[T1]{fontenc}
\usepackage[usenames,dvipsnames]{color}
\usepackage[colorlinks=true,linkcolor=Red,citecolor=Green]{hyperref}
\usepackage[super]{nth}
\usepackage[open, openlevel=2, depth=3, atend]{bookmark}
\hypersetup{pdfstartview=XYZ}
\usepackage[font=footnotesize]{caption}
\usepackage{a4wide}

\usepackage{epstopdf}
 
\usepackage{hyperref}

\theoremstyle{plain}
\newtheorem{theorem}{Theorem}[section]
\newtheorem{lemma}{Lemma}[section]
\newtheorem{proposition}{Proposition}[section]
\newtheorem{corollary}{Corollary}[section]

\theoremstyle{definition}
\newtheorem{definition}{Definition}[section]

\theoremstyle{remark}
\newtheorem{remark}{Remark}[section]
 
\numberwithin{equation}{section}

\newcommand{\R}{\mathbb{R}}

\newcommand{\HH}{\mathbb{H}}
\newcommand{\eps}{\varepsilon}
\newcommand{\ke}{\text{ker }}
\newcommand{\wt}{\widetilde}

\newcommand{\be}{\begin{equation}}
\newcommand{\ee}{\end{equation}}

\DeclarePairedDelimiter\floor{\lfloor}{\rfloor}

% ...

% ...

\title
[Local marked boundary rigidity under hyperbolic trapping assumptions]
{Local marked boundary rigidity under hyperbolic trapping assumptions}

\author{Thibault Lefeuvre}

\address{Laboratoire de Mathématiques d’Orsay, Univ. Paris-Sud, CNRS, Université Paris-Saclay, 91405 Orsay, France}

\email{thibault.lefeuvre@u-psud.fr}

\begin{document}

\begin{abstract}
Under the assumption that the X-ray transform over symmetric solenoidal $2$-tensors is injective, we prove that smooth compact connected manifolds with strictly convex boundary, no conjugate points and a hyperbolic trapped set are \textit{locally marked boundary rigid}. 
\end{abstract}

\maketitle

Throughout this paper, we shall work in the smooth category, that is all the manifolds and coordinate charts are considered to be smooth.

\section{Introduction}

Given $(M,g)$ a compact manifold with boundary of dimension $n \geq 2$, the \textit{marked boundary distance} is defined as the map
\[ d_g : \left\{ (x,y,[\gamma]), (x,y) \in \partial M \times \partial M, [\gamma] \in \mathcal{P}_{x,y} \right\} \rightarrow \R_+ \]
which associates to $x$ and $y$ on the boundary and a homotopy class
\[ [\gamma]  \in \mathcal{P}_{x,y}:=\left\{ [\gamma], \gamma \text{ is a curve joining } x \text{ to } y \right\},\] the distance between $x$ and $y$ computed as the infimum over the piecewise $\mathcal{C}^1$-curves joining $x$ to $y$ in the homotopy class of $[\gamma]$. This map generalizes the classical notion of \textit{boundary distance} to the case of a manifold with topology. It can be seen as an analogue of the \textit{marked length spectrum} in the case of a closed Riemannian manifold, studied for instance in the celebrated articles of Otal \cite{Otal-90} and Croke \cite{Croke-90}.

In the case of a manifold with strictly convex boundary and no conjugate points (which we will consider throughout this paper), there exists a unique geodesic in each homotopy class of curves joining $x$ to $y$ which realizes the distance (see \cite[Lemma 2.2]{Guillarmou-Mazzucchelli-18}). As a consequence, given $[\gamma] \in \mathcal{P}_{x,y}$, $d_g(x,y,[\gamma])$ is nothing but the length of this unique geodesic in the class $[\gamma]$. Given $g'$, another metric with strictly convex boundary and no conjugate points, we will say that their marked boundary distance \textit{agree} if $d_g = d_{g'}$. Note that one can also lift this distance to the universal cover $\wt{M}$ of $M$. Then, there exists a unique geodesic joining any pair of points on the boundary of $\wt{M}$ and the marked boundary distances agree if and only if the two boundary distances $d_{\wt{g}}$ and $d_{\wt{g}'}$ agree.

A classical conjecture in Riemannian geometry is that, under suitable assumptions on the metric, this marked boundary distance determines the metric up to a natural obstruction, in the sense that if $g'$ is another metric with same marked boundary distance function, then there exists a diffeomorphism $\phi : M \rightarrow M$ such that $\phi|_{\partial M} = \text{id}$ and $\phi^*g' = g$. When this occurs, we say that $(M,g)$ is \textit{marked boundary rigid}.

In the case of a \textit{simple} manifold, i.e. a manifold with strictly convex boundary and such that the exponential map is a diffeomorphism at all points (such manifolds are topological balls without trapping and conjugate points), this conjecture was first stated by Michel \cite{Michel-81} in 1981, and later proved by Pestov-Uhlmann \cite{Pestov-Uhlmann-05} in 2002, in the two-dimensional case. It is still an open question in higher dimensions but Stefanov-Uhlmann-Vasy \cite{Stefanov-Uhlmann-Vasy-17} proved the rigidity of a wide range of simple (and also non-simple actually) manifolds satisfying a foliation assumption.

There is actually a long history of results regarding the boundary rigidity question on simple manifolds since the seminal work of Michel. Let us mention the contributions of Gromov \cite{Gromov-83}, for regions of $\R^n$, the original paper of Michel \cite{Michel-81} for subdomains of the open hemisphere and the Besson-Courtois-Gallot theorem \cite{Besson-Courtois-Gallot-95}, which implies the boundary rigidity for regions of $\HH^n$ (see also the survey of Croke \cite{Croke-04}). Still in the simple setting, the \textit{local boundary rigidity} was studied by Croke-Dairbekov-Sharafutdinov in \cite{Croke-Dairbekov-Sharafutdinov-00}, by Stefanov-Uhlmann in \cite{Stefanov-Uhlmann-04} and positive results were obtained. More recently, Burago-Ivanov \cite{Burago-Ivanov-10} proved the local boundary rigidity for metrics close enough to the euclidean metric. But very few papers deal with manifolds with trapping. In that case, the first general results where obtained by Guillarmou-Mazzucchelli \cite{Guillarmou-Mazzucchelli-18} for surfaces, where the local marked boundary rigidity was established under suitable assumptions. One of the main results of this paper is the following marked boundary rigidity result for manifolds of negative curvature, which is a local version of Michel's conjecture.

\begin{theorem}
\label{th0}
Let $(M,g)$ be a compact connected $n$-dimensional manifold with strictly convex boundary and negative curvature. We set $N := \floor*{\frac{n+1}{2}}+1$. Then $(M,g)$ is locally marked boundary rigid in the sense that: for any $\alpha > 0$ arbitrarily small, there exists $\eps > 0$ such that for any metric $g'$ with same marked boundary distance as $g$ and such that $\|g' - g\|_{\mathcal{C}^{N,\alpha}} < \eps$, there exists a smooth diffeomorphism $\phi : M \rightarrow M$, such that $\phi|_{\partial M} = \text{id}$ and $\phi^*g' = g$.
\end{theorem}

We actually prove a refined version of this result, which is detailed in \S\ref{ssect:main}. We stress that the marked boundary distance is the natural object to consider insofar as one can construct examples of surfaces satisfying the assumptions of Theorem \ref{th0} with same boundary distance but different marked boundary distances which are not isometric. Indeed, consider a negatively-curved surface $(M,g)$ whose strictly convex boundary has a single component. We can always choose such a surface so that the distance between two points on the boundary is realized by minimizing geodesics which only visit a neighborhood of this boundary. Thus, any small perturbation of the metric away from the boundary will still provide the same boundary distance function but the metrics will no longer be isometric.

Let us eventually mention that the problem of boundary rigidity is closely related to the \textit{lens rigidity} question, that is the reconstruction of the metric $g$ from the knowledge of the scattering map and the exit time function. This question has been extensively studied in the literature. Among other contributions, let us mention that of Stefanov-Uhlmann \cite{Stefanov-Uhlmann-09}, who prove a \textit{local lens rigidity} result on a non-simple manifold (without the assumption on convexity and with a possible trapped set), which is somehow in the spirit of our article. 

Our proof can be interpreted as a non-trivial inverse function theorem, like in \cite{Croke-Dairbekov-Sharafutdinov-00} or \cite{Stefanov-Uhlmann-04}. Indeed, it can be easily showed that the linearized version of the marked boundary distance problem is equivalent to the injectivity of the X-ray transform $I_2$. The problem here is non-linear, but still local, which allows us to recover some of the features of the linearized problem. The key argument here is a quadratic control of the X-ray transform of the difference of the two metrics $f:=g'-g$ (see Lemma \ref{lem:est1}). We do not choose a normal gauge to make the metrics coincide on the boundary but rather impose a solenoidal gauge (this is made possible thanks to an essential lemma in \cite{Croke-Dairbekov-Sharafutdinov-00}). We stress the fact that this paper partly relies on the study of the X-ray transform carried out in \cite{Lefeuvre-18-1}, which allows a finer control on the regularity of the distributions which are at stake in the last paragraph. This is crucial to apply interpolation estimates to conclude in the end. This former article itself strongly relies on the technical tools introduced in both papers of Guillarmou \cite{Guillarmou-17-1} and \cite{Guillarmou-17-2}, which are based on recent and powerful analytical techniques developed in the framework of hyperbolic dynamical systems (see for instance Dyatlov-Guillarmou \cite{Dyatlov-Guillarmou-16}, Dyatlov-Zworski \cite{Dyatlov-Zworski-16}, Faure-Sjöstrand \cite{Faure-Sjostrand-11}).

\subsection{Preliminaries}

\label{ss:intro}

Let us consider $(M,g)$, a compact connected Riemannian manifold with strictly convex boundary and no conjugate points. We denote by $SM$ its unit tangent bundle, that is
\[ SM = \left\{ (x,v) \in TM, |v|_x = 1 \right\}, \]
and by $\pi_0 : SM \rightarrow M$ the canonical projection. The Liouville measure on $SM$ will be denoted by $d\mu$. The incoming (-) and outcoming (+) boundaries of the unit tangent bundle of $M$ are defined by
\[ \partial_\pm SM = \left\{ (x,v) \in TM, x \in \partial M, |v|_{x} = 1, \mp g_x(v,\nu) < 0 \right\}, \]
where $\nu$ is the outward pointing unit normal vector field to $\partial M$. Note in particular that
\[ S(\partial M) = \overline{\partial_+ SM} \cap \overline{\partial_- SM} \]
If $i : \partial SM \rightarrow SM$ is the embedding of $\partial SM$ into $SM$, we define the measure $d\mu_\nu$ on the boundary $\partial SM$ by
\be d\mu_\nu(x,v) := |g_x(v,\nu)| i^* d\mu (x,v) \ee

$\varphi_t$ denotes the (incomplete) geodesic flow on $SM$ and $X$ the vector field induced on $T(SM)$ by $\varphi_t$. Given each point $(x,v) \in SM$, we define the escape time in positive (+) and negative (-) times by:
\be \begin{array}{c} l_+(x,v) := \sup \left\{ t \geq 0, \varphi_t (x,v) \in SM \right\} \in [0, + \infty] \\
l_-(x,v) := \inf \left\{ t \leq 0, \varphi_t (x,v) \in SM \right\} \in [-\infty, 0] 
\end{array} \ee
We say that a point $(x,v)$ is \textit{trapped in the future} (resp. \textit{in the past}) if $l_+(x,v) = + \infty$ (resp. $l_-(x,v) = -\infty$).

\begin{definition}
The incoming (-) and outcoming (+) tails in $SM$ are defined by:
\[ \Gamma_\mp := \left\{ (x,v) \in SM, l_\pm(x,v) = \pm \infty \right\} \]
They consist of the sets of points which are respectively trapped in the future or the past. The trapped set $K$ for the geodesic flow on $SM$ is defined by:
\be K := \Gamma_+ \cap \Gamma_- = \cap_{t \in \R} \varphi_t(SM) \ee
It consists of the set of points which are both trapped in the future and the past.
\end{definition}

These sets are closed in $SM$ and invariant by the geodesic flow. A manifold is said to be \textit{non-trapping} if $K = \emptyset$. The aim of the present article is precisely to bring new results in the case $K \neq \emptyset$. We also assume that $K$ is hyperbolic, that is there exist some constants $C > 0$ and $\nu > 0$ such that for all $z = (x,v) \in K$, there is a continuous flow-invariant splitting
\be \label{eq:split} T_z(SM) = \R X(z) \oplus E_u(z) \oplus E_s(z), \ee
where $E_s(z)$ (resp. $E_u(z)$) is the \textit{stable} (resp. \textit{unstable}) vector space in $z$, which satisfy
\be \begin{array}{c} |d\varphi_t(z) \cdot \xi|_{\varphi_t(z)} \leq C e^{-\nu t} |\xi|_{z}, ~~ \forall t > 0, \xi \in E_s(z) \\
|d\varphi_t(z) \cdot \xi|_{\varphi_t(z)} \leq C e^{-\nu |t|} |\xi|_{z}, ~~ \forall t < 0, \xi \in E_u(z)\end{array} \ee
The norm, here, is given in terms of the Sasaki metric.

In particular, when $K$ is hyperbolic, the following properties hold (see \cite[Proposition 2.4]{Guillarmou-17-2}):
\begin{proposition}
\label{prop:hyp}
\begin{enumerate}
\item $\mu(\Gamma_- \cup \Gamma_+) = 0$,
\item $\tilde{\mu}(\Gamma_\pm \cap \partial_\pm SM) = 0$, where $\tilde{\mu}$ is the measure on $\partial SM$ induced by the Sasaki metric.
\end{enumerate}
\end{proposition}
Note that usually, $K$ has Hausdorff dimension $\text{dim}_H(K) \in [1,2n-1)$. \\

It is convenient to embed the manifold $M$ into a strictly larger manifold $M_e$, such that $M_e$ satisfies the same properties : it is smooth, has strictly convex boundary and no conjugate points (see \cite{Guillarmou-17-2}, Section 2.1 and Section 2.3). Moreover, this can be done so that the longest connected geodesic ray in $SM_e \setminus SM^\circ$ has its length bounded by some constant $L < + \infty$. As a consequence, the trapped set of $M_e$ is the same as the trapped set of $M$ and the sets $\Gamma_\pm$ are naturally extended to $SM_e$. In the following, for $t \in \R$, $\varphi_t$ will actually denote the extension of $\varphi_t|_{SM}$ to $SM_e$.

\begin{figure}[h!]
\begin{center}

\includegraphics[scale=0.9]{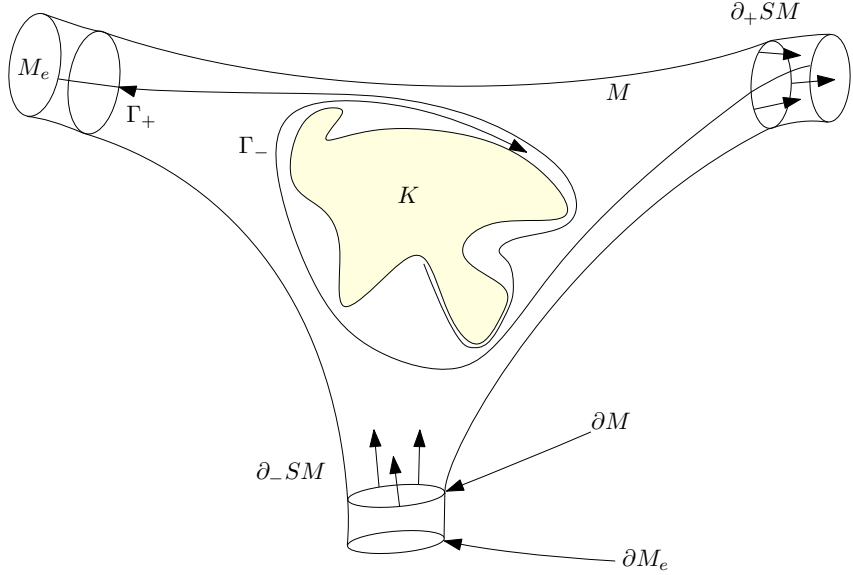} 
\caption{The manifold $M$ embedded in $M_e$}

\end{center}
\end{figure}

\subsection{The X-ray transform}

\label{ssect:xray}

We can now define the X-ray transform:

\begin{definition}
The X-ray transform is the map $I : \mathcal{C}_c^\infty(SM \setminus \Gamma_-) \rightarrow \mathcal{C}^\infty(\partial_-SM)$ defined by:
\[ If(x,v) := \int_0^{+ \infty} f(\varphi_t(x,v)) dt \]
\end{definition}
Note that since $f$ has compact support in the open set $SM \setminus \Gamma_-$, we know that the exit time of any $(x,v) \in SM \setminus \Gamma_-$ is uniformly bounded, so the integral is actually computed over a compact set. It is also natural to extend the action of $I$ on $L^p$-spaces (see \cite[Section 2]{Lefeuvre-18-1}) and one can prove for instance that for any $p > 2$, $I : L^p(SM,d\mu) \rightarrow L^2(\partial_-SM,d\mu_\nu)$ is bounded.

From the definition of $I$, we can define a formal adjoint $I^* : \mathcal{C}_c^\infty(\partial_-SM \setminus \Gamma_-) \rightarrow \mathcal{C}_c^\infty(SM \setminus \Gamma_- \cup \Gamma_+)$ to the X-ray transform by the formula
\be I^*u(x,v) = u(\varphi_{l_-(x,v)} (x,v))  \ee
for the $L^2$ inner scalar products induced by the Liouville measure $d\mu$ on $SM$ and by the measure $d\mu_\nu$ on $\partial_-(SM)$, that is $\langle If, u \rangle_{L^2(\partial_-SM, d\mu_\nu)} = \langle f , I^*u \rangle_{L^2(SM,d\mu)}$, for $f \in \mathcal{C}_c^\infty(SM \setminus \Gamma_-), u \in \mathcal{C}_c^\infty(\partial_-SM \setminus \Gamma_-)$. Note that it naturally extends to a bounded operator $I^* : L^2(\partial_- SM, d\mu_\nu) \rightarrow L^{p'}(SM)$, where $p'$ is the conjugate exponent to $p$ (such that $1/p + 1/p' = 1$).

From this definition of the X-ray transform on functions on $SM$, we can derive the definition of the X-ray transform for symmetric $m$-cotensors. Indeed, such tensors can be seen as functions on $SM$ via the identification map:
\[ \pi_m^* : \left| \begin{array}{l} \mathcal{C}^\infty(M, \otimes_S^m T^*M) \rightarrow \mathcal{C}^\infty(SM) \\ f \mapsto (\pi_m^* f) (x,v) = f(x)(\otimes^m v) \end{array} \right. \]
If $f$ is a (smooth) symmetric $m$-cotensor, its coordinate functions are defined (at least locally) by
\[ f_{i_1,...,i_m}(x) = f(x)(e_{i_1}(x), ..., e_{i_m}(x)), \]
for $1 \leq i_1, ..., i_m \leq n$, where $(e_1(x), ..., e_n(x))$ forms a local orthonormal basis of $T_xM$. The $L^p$-space, for $p \geq 1$, (resp. Sobolev space for $s \geq 0$) of symmetric $m$-cotensors thus consists of tensors whose coordinate functions are all in $L^p(M)$ (resp. $H^s(M)$). An equivalent way to define $H^s(M, \otimes^m_S T^*M)$ is to consider tensors $u$ such that $(1+\Delta)^{s/2} \in L^2(M, \otimes^m_S T^*M)$, where $\Delta = D^*D$ is the Dirichlet Laplacian\footnote{This is an elliptic differential operator with zero kernel and cokernel satisfying the Lopatinskii's transmission condition (see \cite[Theorem 3.3.2]{Sharafutdinov-94}). It can thus be used in order to define the scale of Sobolev spaces.} on $M$ (see below for a definition of $D$ and $D^*$). It is easy to check that $\pi_m^* : L^p(M, \otimes^m_S T^*M) \rightarrow L^p(SM)$ is bounded (resp. $\pi_m^* : H^s(M, \otimes^m_S T^*M) \rightarrow H^s(SM)$).

It also provides a dual operator acting on distributions
\[ {\pi_m}_* : \mathcal{C}^{-\infty}(SM^\circ) \rightarrow \mathcal{C}^{-\infty}(M^\circ, \otimes_S^m T^*M^\circ), \]
such that for $u \in \mathcal{C}^{-\infty}(SM^\circ), f \in \mathcal{C}^\infty(M, \otimes_S^m T^*M), \langle {\pi_m}_* u , f \rangle = \langle u , \pi_m^* f \rangle$, where the distribution pairing is given by the natural scalar product on the bundle $\otimes_S^m T^*M$ induced by the metric $g$, which is written in coordinates, for $f$ and $h$ smooth tensors:
\be \label{eq:ps} \langle f , h \rangle_g = \int_M f_{i_1 ... i_m} g^{i_1 j_1} ... g^{i_m j_m} h_{j_1 ... j_m} d\text{vol} \ee

\begin{definition}
Let $p > 2$ and $p'$ denote its dual exponent such that $1/p + 1/p'=1$. The X-ray transform for symmectric $m$-cotensors is defined by 
\be I_m := I \circ \pi_m^* : L^p(M, \otimes^m_S T^*M) \rightarrow L^2(\partial_-SM, d\mu_\nu) \ee
It is a bounded operator, as well as its adjoint
\be I_m^* = {\pi_m}_* \circ I^* : L^2(\partial_-SM, d\mu_\nu) \rightarrow  L^{p'}(M, \otimes^m_S T^*M) \ee
\end{definition}

Let us now explain the notion of \textit{solenoidal injectivity} of the X-ray transform. If $\nabla$ denotes the Levi-Civita connection and $\sigma : \otimes^{m+1} T^*M \rightarrow \otimes^{m+1}_S T^*M$ the symmetrization operation, we define the \textit{inner derivative} $D = \sigma \circ \nabla : \mathcal{C}^\infty(M, \otimes_S^m T^*M) \rightarrow \mathcal{C}^\infty(M, \otimes_S^{m+1} T^*M)$. The divergence of symmetric $m$-cotensors is its formal adjoint differential operator, given by $D^*f := -\text{tr}_{12}(\nabla f)$, where $\text{tr}_{12} : \mathcal{C}^\infty(M, \otimes_S^m T^*M) \rightarrow \mathcal{C}^\infty(M, \otimes_S^{m-2} T^*M)$ denotes the trace map defined by contracting with the Riemannian metric, namely
\[ \text{tr}_{12}(q)(v_1, ...,v_{m-2}) = \sum_{i=1}^n q(e_i,e_i,v_1,...,v_{m-2}), \]
if $(e_1,...e_n)$ is a local orthonormal basis of $TM$.

If $f \in H^s(M, \otimes^m_S T^*M)$ for some $s \geq 0$, there exists a unique decomposition of the tensor $f$ such that
\[ f = f^s + Dp, \qquad D^*f^s = 0, p|_{\partial M} = 0, \]
where $f^s \in H^s(M, \otimes_S^m T^*M), p \in H^{s+1}(M, \otimes_S^{m-1} T^*M)$ (see \cite[Theorem 3.3.2]{Sharafutdinov-94} for a proof of this result). $f^s$ is called the \textit{solenoidal part} of the tensor whereas $Dp$ is called the \textit{potential part}. Moreover, this decomposition extends to any distribution $f \in H^{-s}(M, \otimes^m_S T^*M)$, $s \geq 0$, as long as it has compact support within $M^\circ$ (see the arguments given in the proof of Lemma \ref{lem:ell} for instance). We will say that $I_m$ is injective over solenoidal tensors, or in short $s$-\textit{injective}, if it is injective when restricted to
\[ \mathcal{C}^\infty_{\text{sol}}(M,\otimes^m_S T^*M) := \mathcal{C}^\infty(M, \otimes_S^m T^*M) \cap \ker D^* \]

This definition stems from the fact that given $p \in \mathcal{C}^\infty(M, \otimes_S^{m-1} T^*M)$ such that $p|_{\partial M} = 0$, one always has $I_m(Dp) = 0$. Indeed $X \pi_m^* = \pi_{m+1}^* D$ and the conclusion is then immediate by the Fundamental Theorem of Calculus. Thus it is morally impossible to recover the potential part of a tensor $f$ without knowing more, except the fact that it lies in the kernel of $I_m$.

\begin{remark}
All these definitions also apply to $M_e$, the extension of $M$. In the following, an index $e$ on an application will mean that it is considered on the manifold $M_e$. The lower indices $\textit{inv}, \textit{comp}, \textit{sol}$ attached to a set of functions or distributions will respectively mean that we consider \textit{invariant} functions (or distributions) with respect to the geodesic flow, \textit{compactly supported} functions (or distributions) within a precribed open set, \textit{solenoidal} tensors (or tensorial distributions).
\end{remark}

\subsection{The normal operator}

Eventually, we define the normal operator $\Pi_m := I_m^* I_m$, for $m \geq 0$. The following result asserts that $\Pi_m$ is a pseudodifferential operator of order $-1$ (this mainly follows from the absence of conjugate points), which is elliptic on $\ke D^*$. It will be at the core of our arguments in \S\ref{sect:tech}. 

\begin{proposition}[\cite{Guillarmou-17-2}, Proposition 5.9]
Under the assumption that $(M,g)$ has no conjugate points and a hyperbolic trapped set, $\Pi_m$ is a pseudodifferential operator of order $-1$ on the bundle $\otimes^m_S T^*M^\circ$ which is elliptic on $\ke D^*$ in the sense that there exists pseudodifferential operators $Q,S,R$ of respective order $1,-2,-\infty$ on $M^\circ$ such that:
\[ Q \Pi_m = \text{id}_{M^\circ} + DSD^* + R \]
\end{proposition}

We will sometimes use this Proposition by adding appropriate cutoff functions: it is actually the way it is stated in \cite{Guillarmou-17-2}. We also refer to \cite{Paternain-Zhou-16} for a proof of the elliptic property and to \cite{Sharafutdinov-94} for the original arguments.

\subsection{Main results}

\label{ssect:main}

We now assume that $(M,g)$ is a compact manifold with strictly convex boundary, no conjugate points and a hyperbolic trapped set. It was proved in \cite[Proposition 2.1]{Guillarmou-Mazzucchelli-18} that there exists $\eps > 0$, such that if $g'$ is another metric satisfying $\|g'-g\|_{\mathcal{C}^2} < \eps$, then $(M,g')$ is a Riemannian manifold with strictly convex boundary, no conjugate points and a hyperbolic trapped set. Note that the proposition is stated in dimension $2$, but the proof is actually independent of the dimension. In the following, we will always assume that $g'$ is close enough to $g$ in the $\mathcal{C}^2$ topology so that it satisfies these assumptions. We introduce $N = \floor*{\frac{n+1}{2}}+1 \geq 2$. We can now state our main result.

\begin{theorem}
\label{th1}
Let $(M,g)$ be a compact connected $n$-dimensional manifold with strictly convex boundary, no conjugate points and hyperbolic trapped set. If $I^e_2$ is $s$-injective on some extension $M_e$ of $M$ (as detailed in \S\ref{ss:intro}), then $(M,g)$ is locally marked boundary rigid in the sense that: for any $\alpha > 0$ arbitrarily small, there exists $\eps > 0$ such that for any metric $g'$ with same marked boundary distance as $g$ and such that $\|g' - g\|_{\mathcal{C}^{N,\alpha}} < \eps$, there exists a smooth diffeomorphism $\phi : M \rightarrow M$, such that $\phi|_{\partial M} = \text{id}$ and $\phi^*g' = g$.
\end{theorem}

In particular, under the assumption that the curvature of $(M,g)$ is non-positive, it was proved in \cite{Guillarmou-17-2} that $I_m$ is $s$-injective for any $m \geq 0$, and thus $m=2$ in particular. This yields a first corollary:

\begin{corollary}
\label{coro1}
Assume $(M,g)$ satisfies the assumptions of Theorem \ref{th1} and has non-positive curvature. Then it is locally marked boundary rigid.
\end{corollary}

Without any assumption on the curvature, we proved in \cite{Lefeuvre-18-1} the $s$-injectivity of $I_2$ for a surface with strictly convex boundary, no conjugate points and hyperbolic trapped set. As a consequence, we recover the following result, which was already proved in \cite{Guillarmou-Mazzucchelli-18} using a different approach.

\begin{corollary}
\label{coro2}
Assume $(M,g)$ is a surface satisfying the assumptions of Theorem \ref{th1}. Then it is locally marked boundary rigid.
\end{corollary}

However, in dimension $n \geq 3$ and without any assumption on the curvature, the injectivity of $I_2$ (and more generally $I_m$, of $m \geq 2$) is still an open question on a manifold satisfying the assumptions of Theorem \ref{th1}.

\subsection{Further developments}

Following a similar framework of proof as the one developed in this article, we were able to prove with Guillarmou in \cite{Guillarmou-Lefeuvre-18} the local rigidity of the marked length spectrum on a closed manifold with negative sectional curvature, thus partially answering a long-standing conjecture of Burns-Katok \cite{Burns-Katok-85}.

\subsection{Acknowledgements}

We warmly thank Colin Guillarmou for fruitful discussions during the redaction of this paper. We are also grateful to the anonymous referee for helpful comments. This project has received funding from the European Research Council (ERC) under the European Union’s Horizon 2020 research and innovation programme (grant agreement No. 725967).

\section{Technical tools}

\label{sect:tech}

We will sometimes drop the notation $C$ for the different constants which may appear at each line of our estimates and rather use the symbol $\lesssim$. By $\|A\| \lesssim \|B\|$, we mean that there exists a constant $C > 0$, which is independent of the elements $A$ and $B$ considered in their respective functional spaces such that, $\|A\| \leq C \|B\|$. In particular, in our case, the constant $C$ will be independent of the tensor $f$. We will also drop the full description of the functional spaces when the context is clear. For the reader's convenience, we hope to simplify the notation by these means. \\

Let us fix some $\eps > 0$ so that any metric $g'$ in an $\eps$-neighborhood of $g$ (with respect to the $\mathcal{C}^2$ topology) is strictly convex, has no conjugate points and a hyperbolic trapped set. We assume from now on that $I^e_2$ is $s$-injective on $M_e$.

\subsection{Reduction of the problem}

It is rather obvious that the metric $g$ is solenoidal with respect to itself since $D^*g = -\text{tr}_{12}(\nabla g) = 0$ ($\nabla g = 0$ since $\nabla$ is the Levi-Civita connection). What is less obvious is that any metric in a vicinity of $g$ is actually isometric to a solenoidal metric (with respect to $g$). We recall that $N = \floor*{\frac{n+1}{2}}+1$.

\begin{proposition}[\cite{Croke-Dairbekov-Sharafutdinov-00}, Theorem 2.1]
There exists a $\mathcal{C}^{N,\alpha}$-neighborhood $W$ of $g$ such that for any $g' \in W$, there exists a $\mathcal{C}^N$-diffeomorphism $\phi : M \rightarrow M$ (it is actually $\mathcal{C}^{N,\alpha}$) preserving the boundary, such that $g'' = \phi^*g'$ is solenoidal with respect to the metric $g$. Moreover, if $W$ is chosen small enough, we can guarantee that $\|g''-g\|_{\mathcal{C}^N} < \eps$.
\end{proposition}

We can thus reduce ourselves to the case where $g'$ is solenoidal with respect to the metric $g$. We introduce $f:=g'-g$, which is, by construction, $\mathcal{C}^N$, solenoidal and satisfies $\|f\|_{\mathcal{C}^N} < \eps$. Our goal is to prove that $f \equiv 0$.

We define $g_\tau := g+\tau f$ for $0 \leq \tau \leq 1$. As mentioned earlier, since $f$ is small enough, each of these metrics have strictly convex boundary, a hyperbolic trapped set and no conjugate points. From now on, we assume that $d_g = d_{g'}$.

\begin{lemma}
$I_2(f) \geq 0$ almost everywhere.
\end{lemma}

\begin{proof}[Proof]
Let $\wt{M}$ denote the universal cover of $M$. We lift all the objects to the universal cover and denote them by $\wt{\cdot}$. We consider $(p,\xi) \in \partial_-S\wt{M} \setminus \wt{\Gamma}_-$ and denote by $q \in \wt{M}$ the endpoint of the geodesic generated by $(p,\xi)$. By \cite[Lemma 2.2]{Guillarmou-Mazzucchelli-18}, we know that for each $\tau \in [0,1]$, there exists a unique $g_\tau$-geodesic $\gamma_\tau : [0,1] \rightarrow \wt{M}$ with endpoints $p$ and $q$. Note that $\gamma_\tau$ depends smoothly on $\tau$\footnote{Indeed, $\tau \mapsto g_\tau$ depends smoothly on $\tau$, so $\xi_\tau := \left(\exp_p^{g_\tau}\right)^{-1}(q)$ depends smoothly on $\tau$. Thus $(t,\tau) \mapsto \varphi_t^{g_\tau}(p,\xi_\tau)$ is smooth in both variables and by the implicit function theorem, the length $l_+^{g_\tau}(p,\xi_\tau)$ is smooth in $\tau$. Thus, the reparametrized geodesic $\gamma_\tau$ depends smoothly on $\tau$.}.

We introduce the energy $E(\tau) := \int_{0}^1 \wt{g}_\tau(\dot{\gamma_\tau}(s),\dot{\gamma_\tau}(s)) ds$. The arguments of \cite[Proposition 3.1]{Croke-Dairbekov-Sharafutdinov-00} apply here as well: they prove that $E$ is a $\mathcal{C}^2$ function on $[0,1]$ which is concave. Moreover, since the boundary distance of $\wt{g}$ and $\wt{g}'$ agree, one has $E(0)=E(1)$. This implies that $E'(0) \geq 0$, but one can see that $E'(0) = \wt{I}_2(\wt{f})(p,\xi)$. Eventually, since $\partial_-S\wt{M}\cap\wt{\Gamma}_-$ has zero measure (with respect to $d\wt{\mu}_\nu$) by Proposition \ref{prop:hyp}, we obtain the result on the universal cover and projecting $\wt{f}$ on the base, we obtain the sought result.
\end{proof}

Notice that, since $\pi_2^*g \equiv 1$ on $SM$, one has for some constant $c_2 > 0$:
\[ \begin{split} \langle g,f \rangle_{L^2(\otimes^2_S T^*M)} & = c_2 \langle \pi_2^*g,\pi_2^*f \rangle_{L^2(SM)} \\ & = c_2 \int_{SM} \pi_2^* f(x,v) d\mu(x,v) \\ & = c_2 \int_{\partial_-SM} I_2(f)(x,v) d\mu_\nu(x,v), \end{split} \]
where the last equality follows from Santaló's formula. But since $I_2(f) \geq 0$ almost everywhere, one gets:
\[ \langle g,f \rangle_{L^2(\otimes^2_S T^*M)} = c_2 \int_{\partial_-SM} I_2(f)(x,v) d\mu_\nu(x,v) = c_2 \|I_2(f)\|_{L^1(\partial_-SM)} \]

We will now prove an estimate on the $L^1$-norm of $I_2(f)$ which is crucial in our proof. It is based on the equality of the volume of $g$ and $g'$, which is a consequence of the fact that their marked boundary distance functions coincide because $\phi$ is isotopic to the identity. Indeed, one can first construct a diffeomorphism $\psi : M \rightarrow M$ such that $\psi|_{\partial M} = \text{id}$ and both $g_0 := \psi^*g$ and $g'$ coincide at all points of $\partial M$ (it is a well-known fact for simple metrics and was proved in \cite[Lemma 2.3]{Guillarmou-Mazzucchelli-18} in our case). Note that $\text{vol}(g_0) = \text{vol}(g)$ and that the marked boundary distance function of $g_0$ and $g'$ still coincide. By \cite[Lemma 2.4]{Guillarmou-Mazzucchelli-18}, this implies that the metrics $g_0$ and $g'$ have same lens data, which, in turn, implies the equality of the two volumes by Santalo's formula (see \cite[Lemma 2.5]{Guillarmou-Mazzucchelli-18}).

\begin{lemma}
\label{lem:est1}
There exists a constant $C > 0$, such that:
\[ \|I_2(f)\|_{L^1(\partial_- SM)} \leq C \|f\|^2_{L^2(M,\otimes^2_S T^*M)} \]
\end{lemma}

\begin{proof}[Proof]
Consider a finite atlas $(U_i,\varphi_i)$ on $M$ and a partition of unity $\sum_i \chi_i = 1$ subordinated to this atlas, i.e. such that $\text{supp}(\chi_i) \subset U_i$. One has for $\tau \in [0,1]$:
\[ \begin{split} \text{vol}(g_\tau) & = \sum_i \int_{\varphi_i(U_i)} \chi_i \circ \varphi_i^{-1} \sqrt{\det(g_\tau(x))} dx, \end{split} \]
where $dx$ denotes the Lebesgue measure and $g_\tau(x)$ the matrix representing the metric in coordinates. In \cite{Croke-Dairbekov-Sharafutdinov-00}, Proposition 4.1, it is proved that for $\|f\|_{C^0} < \eps$ (which is our case), one has pointwise:
\[ \sqrt{\det(g_\tau(x))} \geq \sqrt{\det(g(x))}\left(1+\dfrac{1}{2} \tau \langle g(x),f(x) \rangle_g - \dfrac{1}{4}\tau^2 |f(x)|_g^2-C\eps \tau^3|f(x)|_g^2 \right), \]
where the inner products are computed with respect to the metric, as detailed in (\ref{eq:ps}). Inserting this into the previous integral, we obtain:
\[ \text{vol}(g_\tau) \geq \text{vol}(g) + \dfrac{1}{2}\tau \langle g, f \rangle_{L^2} - \dfrac{1}{4}\tau^2 \|f\|^2_{L^2} - C \eps \tau^3 \|f\|^2_{L^2} \]
Taking $\tau = 1$ and using the fact that $\text{vol}(g') = \text{vol}(g)$, we obtain the sought result.
\end{proof}

\begin{remark}
If $(M,g)$ were a simple manifold, then a well-known Taylor expansion (see \cite[Section 9]{Stefanov-Uhlmann-04} for instance) shows that for $x,y \in \partial M$, one has:
\[ d_{g'}(x,y)=d_g(x,y)+\dfrac{1}{2}I_2(f)(x,y)+R_g(f)(x,y), \]
where $I_2(f)(x,y)$ stands for the X-ray transform with respect to $g$ along the unique geodesic joining $x$ to $y$, $R_g(f)$ is a remainder satisfying:
\[ |R_g(f)(x,y)| \lesssim |x-y| \cdot \|f\|^2_{\mathcal{C}^1(M)} \]
As a consequence, if the two boundary distances agree, one immediately gets that
\[\|I_2(f)\|_{L^\infty(\partial_-SM)} \lesssim \|f\|^2_{\mathcal{C}^1(M)}\]
In our case, because of the trapping issues, $I_2(f)$ is not $L^\infty$ and such an estimate is hopeless. This is why we have to content ourselves with $L^1/L^2$ estimates in Lemma \ref{lem:est1} (and this will be sufficient in the end) but the idea that linearizing the problem brings an inequality with a square is unchanged.
\end{remark}

\subsection{Functional estimates}

Given a tensor $f$ defined on $M$, $E_0f$ denotes its extension by $0$ to $M_e$, whereas $r_Mf$ denotes the restriction to $M$ of a tensor defined on $M_e$. If $f \in H^{1/4}(M,\otimes^2_S T^* M)$, then $E_0f \in H^{1/4}(M_e,\otimes^2_S T^* M_e)$ (see \cite[Corollary 5.5]{Taylor-11}) and we can decompose the extension $E_0f$ into $E_0 f = q +Dp$, where $q \in H^{1/4}_{\text{sol}}(M_e,\otimes^2_S T^* M_e)$ and $p \in H^{5/4}(M_e,\otimes^2_S T^* M_e)$, with $p|_{\partial M_e}=0$.

\begin{lemma}
For any $r \geq 0$, there exists a constant $C > 0$ such that if $f \in H^{1/4}_{\text{sol}}(M,\otimes^2_S T^* M)$:
\[ \|f\|_{H^{-r}(M,\otimes^2_S T^* M)} \leq C \|q\|_{H^{-r}(M,\otimes^2_S T^* M)} \]
\end{lemma}

Actually, this lemma is valid not just for $1/4$ but for any $0 < s < 1/2$. We chose to take a specific $s$ in order to simplify the notations, and because it will be applied for a much regular $f$ which will therefore be in $H^{1/4}$. Note that, from now on, in order to simplify the notations, we will sometimes write $\|T\|_{H^{s}(M)}$ in short, instead of $\|T\|_{H^{s}(M,\otimes^2_S T^* M)}$.

\begin{proof}[Proof]
We argue by contradiction. Assume we can find a sequence of elements $f_n \in H^{1/4}_{\text{sol}}(M,\otimes^2_S T^* M)$ such that:
\[  \|f_n\|_{H^{-r}(M,\otimes^2_S T^* M)} > n \|q_n\|_{H^{-r}(M,\otimes^2_S T^* M)} \]
We can always assume that $\|f_n\|_{H^{1/4}(M)} = 1$ and thus:
\[ \|q_n\|_{H^{-r}(M)} \leq \dfrac{1}{n} \|f_n\|_{H^{-r}(M)} \lesssim \dfrac{1}{n} \|f_n\|_{H^{1/4}(M)} \rightarrow 0 \]
Now, by compactness, we can extract subsequences so that:
\[ \begin{array}{c} f_n \rightharpoonup  f \in H^{1/4}_{\text{sol}}(M,\otimes^2_S T^* M) \\
f_n \rightarrow f ~~ \text{in } L^2(M,\otimes^2_S T^* M) \\
~ \\
p_n \rightharpoonup p \in H^{5/4}(M_e,\otimes^2_S T^* M_e) \\
p_n \rightarrow p ~~ \text{in } H^{1}(M_e,\otimes^2_S T^* M_e) \\
~ \\
q_n \rightharpoonup q \in H^{1/4}_{\text{sol}}(M_e,\otimes^2_S T^* M_e) \\
q_n \rightarrow q ~~ \text{in } L^2(M_e,\otimes^2_S T^* M_e) \end{array}  \]
Remark that the decomposition $E_0 f_n = q_n + Dp_n$ implies, when passing to the limit in $L^2$, that $E_0 f = q + Dp$. Since $\|q_n\|_{H^{-r}(M)} \rightarrow 0$, we have that $q \equiv 0$ in $M$. In $M_e \setminus M$, we have $q = -Dp$. Thus:
\[ 0 = \langle D^*q,p\rangle_{L^2(M_e)} = \langle q,Dp \rangle_{L^2(M_e)} = \langle q,Dp \rangle_{L^2(M_e \setminus M)} = - \|q\|^2_{L^2(M_e \setminus M)}, \]
that is $q \equiv 0$. As a consequence, in $M_e \setminus M^\circ$, $E_0f = 0 = Dp$ and $p|_{\partial M_e} = 0$, so $p \equiv 0$ in $M_e \setminus M^\circ$ by unique continuation. Since $p \in H^{5/4}$, by the trace theorem, we obtain that $p|_{\partial M} = 0$ (in $H^{3/4}(\partial M)$). Since $f$ is solenoidal, $D^*f=0$, and
\[ 0= \langle D^*f,p \rangle_{L^2(M)} = \langle D^*Dp, p \rangle_{L^2(M)} = \|Dp\|^2_{L^2(M)} \]
Therefore, $p \equiv 0$ and, in particular, in $M$, we get that $f=0$ which is contradicted by the fact that $\|f_n\|_{H^{1/4}(M)}=1$.
\end{proof}

We recall that $I^e_2$ is assumed to be injective. Let us mention that if $u \in \mathcal{C}^\infty(M_e,\otimes^2_S T^*M_e)$ is in the kernel of $\Pi^e_2$, then:
\[ 0 = \langle \Pi^e_2 u, u \rangle = \langle {I^e_2}^*I^e_2 u, u \rangle = \|I^e_2 u\|^2, \]
that is $I^e_2 u = 0$. This will be used in the following lemma:

\begin{lemma}
\label{lem:ell}
Under the assumption that $I^e_2$ is injective, for any $r \geq 0$, there exists a constant $C > 0$ such that if $f \in H^{1/4}_{\text{sol}}(M,\otimes^2_S T^* M)$, then:
\[ \|f\|_{H^{-r-1}(M,\otimes^2_S T^* M)} \leq C \|\Pi^e_2 E_0 f\|_{H^{-r}(M_e,\otimes^2_S T^* M_e)} \]
\end{lemma}

\begin{proof}
Let $\chi$ be a smooth positive function supported within $M^\circ_e$ such that $\chi \equiv 1$ in a vicinity of $M$. We know by \cite{Guillarmou-17-2}, that there exists pseudodifferential operators $Q,S,R$ with respective order $1,-2,-\infty$ on $M^\circ_e$ such that:
\[ Q \chi \Pi^e_2 \chi = \chi^2 + D\chi S \chi D^* + R\]
Let us decompose $E_0f = q + Dp$, where $q \in H^{1/4}_{\text{sol}}(M_e,\otimes^2_S T^* M_e)$ and $Dp$ is the potential part given by $p:=\Delta^{-1}D^*E_0f$, $\Delta = D^*D$ being the Laplacian with Dirichlet conditions. Remark that $\chi E_0 f = E_0 f$, and
\[
\begin{split}
r_M Q \chi \Pi^e_2 (E_0 f) & = r_M Q \chi \Pi^e_2 (\chi E_0 f) \\
& = r_M Q \chi \Pi^e_2 \chi (q) + \underbrace{r_M Q \chi \Pi^e_2 D (\chi p)}_{=0} + r_M Q \chi \Pi^e_2 [\chi,D] (p) \\
& = r_M(q) + r_M R (q) + r_M Q \chi \Pi^e_2 [\chi,D] \Delta^{-1}D^*E_0(f)
\end{split}
\]
Note that $[\chi,D]$ is a differential operator supported in the annulus $\left\{ \nabla \chi \neq 0 \right\}$. In particular, $r_M T := r_M Q \chi \Pi^e_2 [\chi,D] \Delta^{-1}D^*E_0 : H^{-r-1} \rightarrow H^{-r-1}$ is a well-defined compact operator on $M$. Using the previous lemma, we obtain:
\[ \begin{split} \|f\|_{H^{-r-1}(M)} & \lesssim \|q\|_{H^{-r-1}(M)} \\
& \lesssim \|r_M Q \chi \Pi^e_2 E_0 f\|_{H^{-r-1}(M)} + \|r_M R q\|_{H^{-r-1}(M)} + \|r_M T f\|_{H^{-r-1}(M)}  \\
& \lesssim \|\Pi^e_2 E_0 f\|_{H^{-r}(M_e)} + \|r_M R q\|_{H^{-r-1}(M)} + \|r_M T f\|_{H^{-r-1}(M)} \end{split} \]
In other words, there exists a constant $C > 0$ such that:
\be \label{eq:cauchy} \|f\|_{H^{-r-1}(M)} \leq C (\|\Pi^e_2 E_0 f\|_{H^{-r}(M_e)} + \|r_M R q\|_{H^{-r-1}(M)} + \|r_M T f\|_{H^{-r-1}(M)}) \ee

The rest of the proof now boils down to a standard argument of functional analysis. Assume by contradiction that we can find a sequence of elements $f_n \in H^{1/4}(M,\otimes^2_S T^* M)$ such that 
\[ \|f_n\|_{H^{-r-1}(M,\otimes^2_S T^* M)} > n \|\Pi^e_2 E_0 f_n\|_{H^{-r}(M_e,\otimes^2_S T^* M_e)} \]
We can always assume that $\|f_n\|_{H^{-r-1}} = 1$ and thus $\|\Pi^e_2 E_0 f_n\|_{H^{-r}} \rightarrow 0$. By construction, $\|q_n\|_{H^{-r-1}} \lesssim \|f_n\|_{H^{-r-1}} = 1$, i.e. $(q_n)$ is bounded in $H^{-r-1}$. Moreover, since $r_M R$ and $r_M T$ are compact, we know that up to a subsequence $r_M R q_n \rightarrow v_1, r_M T f_n \rightarrow v_2$, with $v_1, v_2 \in H^{-r-1}(M,\otimes^2_S T^* M)$. As a consequence, $(r_M R q_n)_{n \geq 0}, (r_M T f_n)_{n \geq 0}$ are Cauchy sequences and applying (\ref{eq:cauchy}) with $f_n-f_m$, we obtain that $(f_n)_{n \geq 0}$ is a Cauchy sequence too. It thus converges to an element $f \in H^{-r-1}_{\text{sol}}(M,\otimes^2_S T^* M)$ which satisfies $\Pi^e_2 E_0 f = 0$. But we claim that $\Pi^e_2 E_0$ is injective on $H^{-r-1}_{\text{sol}}(M,\otimes^2_S T^* M)$. Assuming this claim, this implies that $f = 0$, which contradicts the fact that $\|f_n\|_{H^{-r-1}} = 1$.

Let us now prove the injectivity. It is the exact same argument as the one given in \cite[Lemma 2.6]{Lefeuvre-18-1} but we reproduce it here for the reader's convenience. Assume $\Pi^e_2 E_0 f = 0$ for some $f \in H^{-r-1}_{\text{sol}}(M,\otimes^2_S T^* M)$. Since $E_0f$ has compact support within $M^\circ_e$, we can still make sense of the decomposition $E_0f = q +Dp$, where $p:=\Delta^{-1}D^*E_0f \in H^{-r}(M_e,\otimes^2_S T^* M_e)$, $\Delta := D^*D$ is the Laplacian with Dirichlet conditions and $q:=E_0f-Dp \in H^{-r-1}_{\text{sol}}(M_e,\otimes^2_S T^* M_e)$ (in the sense that $D^*q = 0$ in the sense of distributions). By ellipticity of $\Delta$, $p$ has singular support contained in $\partial M$ (since $\Delta p = D^*E_0f$), and the same holds for $Dp$. Moreover: 
\[ \Pi^e_2(E_0f) = 0 = \Pi^e_2(q) + \Pi^e_2(Dp) = \Pi^e_2(q) \]
From $q=-Dp$ on $M_e \setminus M$, we see that $q$ is smooth on $M_e \setminus M$ and since it is solenoidal on $M_e$ and in the kernel of $\Pi^e_2$, it is smooth on $M_e^\circ$ (this stems from the ellipticity of $\Pi^e_2$ on $\ke D^*$). As a consequence, $q \in \mathcal{C}^\infty_{\text{sol}}(M_e,\otimes^2_S T^*M_e) \cap \ker I^e_2$ and thus $q=0$ by $s$-injectivity of the X-ray transform. We have $E_0f = Dp$ and $E_0f = 0$ on $M_e \setminus M$, $p|_{\partial M_e} = 0$. By unique continuation, we obtain that $p = 0$ in $M_e \setminus M$. Now, by ellipticity, one can also find pseudo-differential operators $Q, S, R$ on $M_e^\circ$ of respective order $1,-2,-\infty$, such that:
\[  Q \Pi^e_2 = \text{id}_{M^\circ_e} +  D S D^* + R, \]
where $S$ is a parametrix of $D^*D$. Since $E_0f = Dp$ has compact support in $M^\circ_e$, we obtain:
\[ \begin{split} Q \Pi^e_2 E_0 f & = 0 \\
& = Q \Pi^e_2 Dp \\
& = Dp + D S D^* Dp + Rp \\
& = 2Dp + \text{ smooth terms} \end{split} \]
This implies that $E_0f = Dp$ is smooth on $M_e$, vanishes on $\partial M$. Therefore:
\[ \langle f, f \rangle_{L^2(M)}  = \langle f , Dp \rangle = \langle D^* f , p \rangle = 0, \]
that is $f \equiv 0$.
\end{proof}

For $s \in \R$, we define $H^{s}_{\text{inv}}(SM)$ to be the set of $u \in H^{s}(SM)$ such that $X u = 0$ (in the sense of distributions if $s < 1$). The following lemma will allow us some gain in the "battle" of exponents in the proof of the Theorem. 

\begin{lemma}
\label{lem:gego}
For all $s \in \R$, $m \geq 0$,
\[{\pi_m}_* :  H^{s}_{\text{inv}}(SM) \rightarrow H^{s+1/2}(M, \otimes^m_S T^* M)\]
is bounded (and the same result holds for $M_e$).
\end{lemma}

\begin{proof}[Proof]
We fix $s \in \R$. The idea is to see ${\pi_m}_*$ as an averaging operator in order to apply Gérard-Golse's result of regularity (\cite[Theorem 2.1]{Gerard-Golse-92}). In local coordinates, given $f \in \mathcal{C}^\infty(SM)$, one has (see \cite[Section 2]{Paternain-Zhou-16} for instance) :
\[ {\pi_m}_* f(x)_{i_1 ... i_m} = g_{i_1 j_1}(x) ... g_{i_m j_m}(x) \int_{S_xM} f(x,\xi) \xi^J dS_x(\xi), \]
where $\xi^J = \xi^{j_1} ... \xi^{j_m}$. It is thus sufficient to prove that the $H^{s+1/2}$-norm of each of these coordinates is controlled by the $H^s$-norm of $f$. Since $(M,g)$ is smooth, it is actually sufficient to control the $H^{s+1/2}$-norm of the integral. Note that
\be \label{eq:c2} \|X (f\xi^J) \|_{H^s(SM)} \lesssim \|f\|_{H^s(SM)} + \|Xf\|_{H^s(SM)} \ee
Since $X$ satisfies the transversality assumption of \cite[Theorem 2.1]{Gerard-Golse-92}, we conclude that $u : x \mapsto \int_{S_xM} f(x,\xi) \xi^J dS_x(\xi)$ is in $H^{s+1/2}(M)$. By (\ref{eq:c2}), we also know that its $H^{s+1/2}$-norm is controlled by
\be
\label{eq:control}
\|u\|_{H^{s+1/2}(SM)} \lesssim \|f\|_{H^s(SM)} + \|Xf\|_{H^s(SM)}.
\ee

Now, if $f \in H^s_{\text{inv}}(SM)$, there exists by \cite[Lemma E.47]{Dyatlov-Zworski-book-resonances} a sequence of smooth functions $f_n \in \mathcal{C}^\infty(SM)$ such that $f_n \rightarrow f, Xf_n \rightarrow Xf=0$ in $H^s(SM)$. We obtain the sought result by passing to the limit in (\ref{eq:control}).
\end{proof}

We will apply this lemma with $m=2$. The following result is proved in \cite[Proposition 2.2]{Lefeuvre-18-1}:

\begin{lemma}
\label{lem:est3}
Let $1 < q < p < + \infty$. Then $I : L^p(SM) \rightarrow L^q(\partial_-SM)$ and $I^* : L^p(\partial_-SM) \rightarrow L^q(SM)$ are bounded, and the same statement holds for $M_e$.
\end{lemma}

Eventually, the following lemma is stated for Sobolev spaces in \cite[Lemma 6.2]{Paternain-Zhou-16}, but the same result holds for Lebesgue spaces. The proof relies on the fact that, by construction of the extension $M_e$, there exists a maximum time $L < +\infty$ for a point in $\partial_-SM_e$ to either exit $SM_e$ or to hit $\partial_-SM$.

\begin{lemma}
\label{lem:est2}
Let $1 \leq p < + \infty$. There exists a constant $C > 0$ such that if $f \in L^1(M,\otimes^2_S T^*M)$ is a section such that $I_2(f) \in L^p(\partial_-SM)$ and $E_0f$ denotes its extension by $0$ to $M_e$, one has:
\be \|I^e_2(E_0f)\|_{L^p(\partial_-SM_e)} \leq C \|I_2(f)\|_{L^p(\partial_-SM)} \ee
\end{lemma}

\section{End of the proof}

We now have all the ingredients to conclude the proof of Theorem \ref{th1}. Note that there are arbitrary choices made as to the functional spaces considered. The bounds we obtain are clearly not optimal, but this is of no harm as to the content of the theorem. In particular, we are limited by the Sobolev injection used in the proof, which depends on the dimension: this is why we loose regularity in the theorem as the dimension increases.

\begin{proof}[Proof of the Theorem]
We already know by Lemma \ref{lem:est1} that
\[ \|I_2(f)\|_{L^1(\partial_-SM)} \lesssim \|f\|^2_{L^2(M,\otimes^2_S T^* M)}\]
We recall that $N = \floor*{\frac{n+1}{2}}+1 > \frac{n+1}{2}$. We fix $q \in (1,2)$ close to $1$ and set $s = n\left(\frac{1}{q}-\frac{1}{2}\right)$, the exponent of the Sobolev injection $L^q \xhookrightarrow{} H^{-s}$. Interpolating $L^2$ between the Sobolev spaces $H^{-s-1/2}$ and $H^N$, we obtain for $\gamma = \frac{N}{s+1/2+N}$:
\[ \|I_2(f)\|_{L^1(\partial_-SM)} \lesssim \|f\|_{L^2}^2 \lesssim \|f\|_{H^{-s-1/2}}^{2\gamma} \|f\|_{H^N}^{2(1-\gamma)}  \lesssim \|f\|_{H^{-s-1/2}}^{2\gamma} \|f\|_{\mathcal{C}^N}^{2(1-\gamma)} \]
Moreover, by Lemma \ref{lem:est3}, we have that for $p > 1$ large enough and for $\delta > 0$ as small as wanted, $\|I_2(f)\|_{L^p(\partial_-SM)} \lesssim \|f\|_{L^{p+\delta}(M,\otimes^2_S T^* M)} \lesssim \|f\|_{L^\infty(M,\otimes^2_S T^* M)}$. By interpolation, we obtain that:
\[ \begin{split} \|I_2(f)\|_{L^{q+\delta}(\partial_-SM)} & \lesssim \|I_2(f)\|^{\theta}_{L^1}\|I_2(f)\|^{1-\theta}_{L^{p}} \\ & \lesssim \|f\|^{2\theta}_{L^2} \|f\|^{1-\theta}_{L^{\infty}} \\
& \lesssim  \|f\|_{H^{-s-1/2}}^{2\gamma \theta} \|f\|_{\mathcal{C}^N}^{2(1-\gamma)\theta} \|f\|^{1-\theta}_{L^{\infty}} ,\end{split}\]
where $\theta \in [0,1]$ satisfies 
\be \label{eq:th} \dfrac{1}{q + \delta} = \theta + \dfrac{1-\theta}{p}\ee

As a consequence, we obtain:
\[ \begin{array}{lll} \|f\|_{H^{-s-1/2}} & \lesssim \|\Pi_2^e E_0 f\|_{H^{-s+1/2}} & \text{ by Lemma } \ref{lem:ell} \\ 
& \lesssim \|{I^e}^*I^e_2E_0f\|_{H^{-s}}& \text{ by Lemma } \ref{lem:gego} \\
& \lesssim \|{I^e}^* I^e_2 E_0 f\|_{L^{q}} & \text{ by Sobolev injection } L^{q} \xhookrightarrow{} H^{-s}\\
& \lesssim \|I^e_2 E_0 f\|_{L^{q+\delta}}& \text{ by Lemma } \ref{lem:est3}\\
& \lesssim \|I_2 f\|_{L^{q+\delta}} & \text{ by Lemma } \ref{lem:est2}\\
& \lesssim \|f\|_{H^{-s-1/2}}^{2\gamma \theta} \|f\|_{\mathcal{C}^N}^{2(1-\gamma)\theta} \|f\|^{1-\theta}_{L^{\infty}}& \end{array} \]
Remark that we can choose $q$ as close we want to $1$, thus $s$ close enough to $n/2$ and $\theta$ close enough to $1/q$. In the limit $q=1,s=n/2,\hat{\theta}=1/q, \hat{\gamma}=\frac{N}{n/2+1/2+N}$, we have:
\[ 2\hat{\gamma}\hat{\theta} = \dfrac{2N}{n/2 + 1/2 + N} > 1, \]
since $N = \floor*{\frac{n+1}{2}}+1 > \frac{n+1}{2}$. As a consequence, we can always make some choice of constants $q,p,\delta$ which guarantees that $2\gamma\theta > 1$. Now, if $f$ were not zero, one would obtain:
\[ C \leq \|f\|_{H^{-s-1/2}}^{2\gamma \theta -1} \|f\|_{\mathcal{C}^N}^{2(1-\gamma)\theta} \|f\|^{1-\theta}_{L^{\infty}} \leq C' \eps^{\theta}, \]
for some constants $C$ and $C'$, independent of $f$, and we get a contradiction, provided $\eps$ is chosen small enough at the beginning.

As a consequence, for $g'$ smooth with same marked boundary distance and such that $\|g'-g\|_{\mathcal{C}^N} < \eps$, there exists a $\mathcal{C}^N$-diffeomorphism which preserves the boundary and such that $\phi^* g' = g$. Note that both $g$ and $g'$ are smooth : it is a classical fact that such an isometry $\phi$ is actually smooth.

\end{proof}

\bibliographystyle{alpha}
\bibliography{biblio}

\begin{thebibliography}{{Lef}18}

\bibitem[BCG95]{Besson-Courtois-Gallot-95}
G.~Besson, G.~Courtois, and S.~Gallot.
\newblock Entropies et rigidit\'{e}s des espaces localement sym\'{e}triques de
  courbure strictement n\'{e}gative.
\newblock {\em Geom. Funct. Anal.}, 5(5):731--799, 1995.

\bibitem[BI10]{Burago-Ivanov-10}
Dmitri Burago and Sergei Ivanov.
\newblock Boundary rigidity and filling volume minimality of metrics close to a
  flat one.
\newblock {\em Ann. of Math. (2)}, 171(2):1183--1211, 2010.

\bibitem[BK85]{Burns-Katok-85}
K.~Burns and A.~Katok.
\newblock Manifolds with nonpositive curvature.
\newblock {\em Ergodic Theory Dynam. Systems}, 5(2):307--317, 1985.

\bibitem[CDS00]{Croke-Dairbekov-Sharafutdinov-00}
Christopher~B. Croke, Nurlan~S. Dairbekov, and Vladimir~A. Sharafutdinov.
\newblock Local boundary rigidity of a compact {R}iemannian manifold with
  curvature bounded above.
\newblock {\em Trans. Amer. Math. Soc.}, 352(9):3937--3956, 2000.

\bibitem[Cro90]{Croke-90}
Christopher~B. Croke.
\newblock Rigidity for surfaces of nonpositive curvature.
\newblock {\em Comment. Math. Helv.}, 65(1):150--169, 1990.

\bibitem[Cro04]{Croke-04}
Christopher~B. Croke.
\newblock Rigidity theorems in {R}iemannian geometry.
\newblock In {\em Geometric methods in inverse problems and {PDE} control},
  volume 137 of {\em IMA Vol. Math. Appl.}, pages 47--72. Springer, New York,
  2004.

\bibitem[DG16]{Dyatlov-Guillarmou-16}
Semyon Dyatlov and Colin Guillarmou.
\newblock Pollicott-{R}uelle resonances for open systems.
\newblock {\em Ann. Henri Poincar\'{e}}, 17(11):3089--3146, 2016.

\bibitem[DZ]{Dyatlov-Zworski-book-resonances}
Semyon Dyatlov and Maciej Zworski.
\newblock {\em Mathematical Theory of Resonances}.
\newblock \href{Version preliminaire}{math.mit.edu/~dyatlov/res/}, **.

\bibitem[DZ16]{Dyatlov-Zworski-16}
Semyon Dyatlov and Maciej Zworski.
\newblock Dynamical zeta functions for {A}nosov flows via microlocal analysis.
\newblock {\em Ann. Sci. \'{E}c. Norm. Sup\'{e}r. (4)}, 49(3):543--577, 2016.

\bibitem[FS11]{Faure-Sjostrand-11}
Fr\'{e}d\'{e}ric Faure and Johannes Sj\"{o}strand.
\newblock Upper bound on the density of {R}uelle resonances for {A}nosov flows.
\newblock {\em Comm. Math. Phys.}, 308(2):325--364, 2011.

\bibitem[GG92]{Gerard-Golse-92}
Patrick G\'{e}rard and Fran\c{c}ois Golse.
\newblock Averaging regularity results for {PDE}s under transversality
  assumptions.
\newblock {\em Comm. Pure Appl. Math.}, 45(1):1--26, 1992.

\bibitem[GL18]{Guillarmou-Lefeuvre-18}
C.~{Guillarmou} and T.~{Lefeuvre}.
\newblock {The marked length spectrum of Anosov manifolds}.
\newblock {\em ArXiv e-prints}, June 2018.

\bibitem[GM18]{Guillarmou-Mazzucchelli-18}
Colin Guillarmou and Marco Mazzucchelli.
\newblock Marked boundary rigidity for surfaces.
\newblock {\em Ergodic Theory Dynam. Systems}, 38(4):1459--1478, 2018.

\bibitem[Gro83]{Gromov-83}
Mikhael Gromov.
\newblock Filling {R}iemannian manifolds.
\newblock {\em J. Differential Geom.}, 18(1):1--147, 1983.

\bibitem[Gui17a]{Guillarmou-17-1}
Colin Guillarmou.
\newblock Invariant distributions and {X}-ray transform for {A}nosov flows.
\newblock {\em J. Differential Geom.}, 105(2):177--208, 2017.

\bibitem[Gui17b]{Guillarmou-17-2}
Colin Guillarmou.
\newblock Lens rigidity for manifolds with hyperbolic trapped sets.
\newblock {\em J. Amer. Math. Soc.}, 30(2):561--599, 2017.

\bibitem[{Lef}18]{Lefeuvre-18-1}
Thibault {Lefeuvre}.
\newblock {On the s-injectivity of the X-ray transform on manifolds with
  hyperbolic trapped set}.
\newblock {\em ArXiv e-prints}, page arXiv:1807.03680, July 2018.

\bibitem[Mic82]{Michel-81}
Ren\'{e} Michel.
\newblock Sur la rigidit\'{e} impos\'{e}e par la longueur des
  g\'{e}od\'{e}siques.
\newblock {\em Invent. Math.}, 65(1):71--83, 1981/82.

\bibitem[Ota90]{Otal-90}
Jean-Pierre Otal.
\newblock Le spectre marqu\'{e} des longueurs des surfaces \`a courbure
  n\'{e}gative.
\newblock {\em Ann. of Math. (2)}, 131(1):151--162, 1990.

\bibitem[PU05]{Pestov-Uhlmann-05}
Leonid Pestov and Gunther Uhlmann.
\newblock Two dimensional compact simple {R}iemannian manifolds are boundary
  distance rigid.
\newblock {\em Ann. of Math. (2)}, 161(2):1093--1110, 2005.

\bibitem[PZ16]{Paternain-Zhou-16}
Gabriel~P. Paternain and Hanming Zhou.
\newblock Invariant distributions and the geodesic ray transform.
\newblock {\em Anal. PDE}, 9(8):1903--1930, 2016.

\bibitem[Sha94]{Sharafutdinov-94}
V.~A. Sharafutdinov.
\newblock {\em Integral geometry of tensor fields}.
\newblock Inverse and Ill-posed Problems Series. VSP, Utrecht, 1994.

\bibitem[Shu01]{Shubin-01}
M.~A. Shubin.
\newblock {\em Pseudodifferential operators and spectral theory}.
\newblock Springer-Verlag, Berlin, second edition, 2001.
\newblock Translated from the 1978 Russian original by Stig I. Andersson.

\bibitem[SU04]{Stefanov-Uhlmann-04}
Plamen Stefanov and Gunther Uhlmann.
\newblock Stability estimates for the {X}-ray transform of tensor fields and
  boundary rigidity.
\newblock {\em Duke Math. J.}, 123(3):445--467, 2004.

\bibitem[SU09]{Stefanov-Uhlmann-09}
Plamen Stefanov and Gunther Uhlmann.
\newblock Local lens rigidity with incomplete data for a class of non-simple
  {R}iemannian manifolds.
\newblock {\em J. Differential Geom.}, 82(2):383--409, 2009.

\bibitem[SUV17]{Stefanov-Uhlmann-Vasy-17}
Plamen {Stefanov}, Gunther {Uhlmann}, and Andras {Vasy}.
\newblock {Local and global boundary rigidity and the geodesic X-ray transform
  in the normal gauge}.
\newblock {\em ArXiv e-prints}, page arXiv:1702.03638, February 2017.

\bibitem[Tay11]{Taylor-11}
Michael~E. Taylor.
\newblock {\em Partial differential equations {I}. {B}asic theory}, volume 115
  of {\em Applied Mathematical Sciences}.
\newblock Springer, New York, second edition, 2011.

\end{thebibliography}

\end{document}